\documentclass[12pt]{article}
\usepackage{amsmath}
\usepackage{amssymb}
\usepackage{amsthm}

\newtheorem{thm}{Theorem}
\newtheorem{prop}[thm]{Proposition}
\newtheorem{lemm}[thm]{Lemma}

\newtheorem{dfn}[thm]{Definition}

\def\R{\mathbb{R}}

\def\epar{\partial_{\varepsilon}}
\def\pa{\partial }
\def\dom{\mathrm{dom}\,}

\def\hyp{\mathrm{hyp}\,}
\def\co{\mathrm{co}\,}
\def\e{\varepsilon}
\def\ffi{\varphi}

\title{New Multirectional Mean Value Inequality\footnote{Partially supported by Bulgarian National Scientific Fund under Grant DFNI-I02/10.}}
\author{M. Hamamdjiev, M. Ivanov}
\date{January, 2017}

\begin{document}

\maketitle

\begin{abstract}
  We establish new and stronger inequality of Clarke-Ledyaev type by direct construction.
\end{abstract}

\textbf{2010 AMS Subject Classification: 49K27, 49J53, 49J52.}

\textbf{Keywords and phrases: Multidirectional Mean Value Theorem, Subdifferential.}

	\section{Introduction.}
    In theories on the notion subdifferential it is often cumbersome to list for which of the many subdifferntials a given statement holds. A way around this issue, proposed by Ioffe and established by Thibault and others, is to consider the notion of subdifferential abstractly: as defined by set of axioms rather than by construction. Some of these axioms will be principal, like (P1), (P2) and (P3) below, and some technical, like (P4).

    As it is  well known, a subdifferential operator $\partial$ applied to a lower semicontinuous function $f:X\to\mathbb{R}\cup\{\infty\}$ on a Banach space $X$ produces a multivalued map
    $$
      \pa f:X\to 2^{X^*}.
    $$

		\begin{dfn}\label{def:subdif}
      We call the subdifferential $\pa$ \emph{feasible} if the following properties hold:
      \begin{enumerate}
        \item[\rm{(P1)}] $\pa f(x)=\emptyset$ if $x\not\in\dom f$.
        \item[\rm{(P2)}] $\pa f(x) = \pa g(x)$ whenever $f$ and $g$ coincide in a neighbourhood of $x$.
        \item[\rm{(P3)}] If $f$ is convex and continuous in a neighbourhood of $x$ then $\pa f(x)$ coincides with the standard subdifferential in Convex Analysis.
        \item[\rm{(P4)}] If $g$ is convex and continuous in a neighbourhood of $z$ and $f+g$ has local minimum at $z$ then for each $\e>0$ there are $p\in \pa f(x)$ and $q\in \pa g(y)$ such that
      $$
        \|x-z\|<\e,\ |f(x)-f(z)|<\e,\ \|y-z\|<\e,\mbox{ and }\|p+q\|<\e.
      $$
    \end{enumerate}
    \end{dfn}

		We discuss these axioms in Section~\ref{sec:axioms}. There we point out that most of the known subdifferentials satisfy them under natural conditions on the space.

		Here we state our main result. Let $B_\delta:=B+\delta B_X$, where $B$ is any subset of $X$ and $B_X$ is the closed unit ball. For $A,B\subseteq X$ let $[A,B]$ be the convex hull of $A$ and $B$.

		\begin{thm}\label{thm:cl}
		Let $X$ be a Banach space and let $\pa$ be a feasible subdifferential.
      Let $A$ and $B$ be non-empty closed, bounded and convex subsets of $X$. Let $f:X\to\R\cup\{\infty\}$ be a proper lower semicontinuous function
      such that $A\cap \dom f\neq\emptyset$.
      Let $f$ be bounded below on $C:=\overline{[A,B]_\delta}$ for some $\delta>0$. Let
      \begin{equation}\label{eq:D}
        \mu < \inf_C f.
      \end{equation}
      Let $r,s\in\R$ be such that
			\begin{equation}\label{eq:rs}
        r = \inf_Af,\ s<\inf_{B_\delta}f.
      \end{equation}
			Then for each $\e>0$ there are $\xi\in [A,B]_\delta$ and $p\in\pa f(\xi)$ such that
      \begin{equation}\label{eq:f_localisation2}
        f(\xi) < \inf_{[A,B]}f+|r-s|+\e
      \end{equation}
			\begin{equation}\label{eq:norm_p}
        \|p\| < \frac{\max\{r,s\}-\mu}{\delta}+\e,
      \end{equation}
			and
			\begin{equation}\label{eq:mvi}
        \inf_B p - \inf_A p > s-r.
      \end{equation}
		\end{thm}

		We will compare our result to what is known. Historically, the original multidirectional inequality can be found in \cite{cl1}. It compares the values of a locally Lipschitz function on two bounded, closed and convex subsets of a Banach space, one of which is compact. The next work in the field is \cite{cl2}. There we can find a multidirectional inequality, which compares the values of a lower semicontinuous and bounded below function on a point and a closed, convex and bounded set in the setting of a Hilbert space by using the proximal subgradient. In \cite{zhu1} we find a multidirectional inequality on $\beta$-smooth Banach spaces in a configuration of a point and a closed, bounded and convex set, a lower semicontinuous bounded below function and the corresponding $\beta$ subdifferential. There are a number of subsequent developments, for example \cite{iz2}, \cite{lez}, where we can find different kind of relaxations: non-convexness of a set, function not bounded below, etc.\\

		In this work, one can see that compared to the results from \cite{cl1},\cite{cl2}, \cite{zhu1} we obtain an inequality for two sets for a lower semicontinuous bounded below function on a Banach space and a feasible subdifferential. Moreover, in the conditions of \cite{cl1} we have a stronger inequality, as
		$$
			\inf_Bf-\sup_Af\le\inf_Bf-\inf_Af
		$$
		Furthermore, from the construction is clear that our inequality \eqref{eq:norm_p} is close to the optimal.

    The main tool for our proof is the function $\ffi_K$ constructed in Section~\ref{sec:constr}. By sketching the graph of $\ffi_K$ the reader would readily grasp the idea.

    The reminder of the article is organised as follows: In Section~\ref{sec:axioms} we discuss the axioms of subdifferential. Section~\ref{sec:constr} is devoted to the main construction. Finally, the proof of Theorem~\ref{thm:cl} is presented in Section~\ref{sec:proof}.

  \section{On axioms of subdifferential.}\label{sec:axioms}
    In this section we discuss Definition~\ref{def:subdif}.

    First, observe that (P1), (P2) and (P3) are very common. The form here is essentially that found in \cite{tz}. The fuzzy sum rule (P4) here is formally stronger than the corresponding one in \cite{tz}, but we have no example showing that it is actually stronger. This might be an interesting open question.

    It immediately follows from Corollary~4.64 \cite[p.305]{penot}  that the smooth subdifferential in smooth Banach space (the types of smoothness synchronized, of course) is feasible. For the experts it would be obvious that (P4) has the typical "fuzzy" touch of smooth subdifferential. Indeed, (P4) is modeled  after the corresponding property for smooth subdifferentials.

    Corollary~5.52 \cite[p.385]{penot} and Theorem~7.23 \cite[p.475]{penot} imply that Clarke subdifferential and G-subdifferential of Ioffe satisfy more than (P4), that is, if $z$ is an local minimum of $f+g$ then
    $$
      0\in \pa f(z) + \pa g(z).
    $$
    Similarly, it is easy to show that the limiting subdifferential on Asplund space, considered by Morduchovich and others, is feasible.

    This means that most of the subdifferentials are feasible under natural assumptions on their underlying spaces. And we put 'most' here just to be on the safe side, because we know of no meaningful counterexample. This implies that the result of this paper is very general.

    For simplicity we will derive an equivalent form of (P4) which encapsulates the standard application of Ekeland Variational principle.

    \begin{prop}\label{p4p}
      Let $X$ be a Banach space. Let $f:X\to\mathbb{R}\cup\{\infty\}$ be a proper and lower semicontinuous function.

      If the subdifferential $\partial$ is feasible then it satisfies
      \begin{enumerate}
        \item[$(\mathrm{P4}')$] If $g$ is convex and continuous and if $f+g$ is bounded below, then there are $p_n\in\pa^- f(x_n)$ and $q_n\in \pa^- g(y_n)$ such that
        $$
          \|x_n-y_n\|\to 0,\ (f(x_n)+g(y_n)) \to \inf (f+g),\ \|p_n+q_n\|\to 0.
        $$
      \end{enumerate}
    \end{prop}
    \begin{proof}
      Let $(z_n)_1^\infty$ be a minimizing sequence for $f+g$, that is,
      $f(z_n)+g(z_n)<\inf (f+g) + \e _n$, for some $\e_n\to 0$. By Ekeland Variational Principle, see e.g. Theorem~1.88 \cite[p.62]{penot}, there are $u_n\in X$ such that $\|u_n-z_n\|\le\e_n$ and
      $$
        f(x)+g(x)+\e_n\|x-u_n\|\ge f(u_n)+g(u_n).
      $$
      Set $\hat g_n(x):=g(x)+\e_n\|x-u_n\|$. Since $f+\hat g_n$ has a minimum at $u_n$ and $\hat g_n$ is convex and continuous, by (P4) there are $p_n\in \pa f(x_n)$ and $\hat q_n\in \pa \hat g_n(y_n)$ such that $\|x_n-u_n\| < \e_n$, $|f(x_n)-f(u_n)|<\e_n$, $\|y_n-u_n\|<\e_n$ and $\|p_n+\hat q_n\|<\e_n$.

      All conclusions of $(\mathrm{P4}')$ except the last one follow from triangle inequality.

      For the last one we note that any subdifferential of $\e_n\|\cdot-u_n\|$ is of norm less or equal to $\e_n$. By Sum Theorem of Convex Analysis, see e.g. \cite[p. 206]{penot}, there is $q_n\in \pa g(y_n)$ such that $\|q_n-\hat q_n\|\le \e_n$. We have $\|p_n+q_n\|\le 2\e_n\to 0$.
    \end{proof}

  \section{Main construction.}\label{sec:constr}
	 We will recall few notions.

        Let $f$ and $g$ be functions defined on the Banach space $X$.
	The supremal convolution (sup-convolution) of $f$ and $g$ is the function $f*g$ defined by
	$$
    f*g (x) :=\sup \{f(x-y)+g(y):\ y\in X\}
  $$
	The $\e$-subdifferential and $\e$-superdifferential of a function $f$ are:
  $$
    \epar^- f(x) := \{p\in X^*:\ p(\cdot - x) \le f(\cdot)-f(x) + \e\},
	$$
	$$
		\epar^+ f(x) := \{p\in X^*:\ p(\cdot - x) \ge f(\cdot)-f(x) - \e\},
  $$
	respectively. Note that we will use $\pa$ instead of $\pa_0$ as usual.

	The hypograph of a function $f$ is
	$$
	  \hyp f:=\{(x,t):\ t\le f(x)\}\subseteq X\times\mathbb{R}.
	$$
	Note also that the convex hull of two convex sets $A$ and $B$ can be written as:
	$$
	  [A,B]=\{z=\lambda u+(1-\lambda)v:u\in A, v\in B, \lambda\in[0,1]\}.
	$$
  \begin{lemm}\label{sup_conv}
		If $p\in \pa^+ (f*g)(x)$ and $y\in X$ is such that
		\begin{equation}\label{eq:1}
      f*g(x) -\e \le f(x-y)+g(y),
    \end{equation}
		then $p\in \partial^+_{\e}g(y)$.
	\end{lemm}
	\begin{proof}
		By the definition of sup-convolusion for all $h\in X$
		\begin{eqnarray*}
      f*g(x+h) &\ge& f((x+h)-(y+h))+g(y+h)\\
               &=&   f(x-y) + g(y+h).
    \end{eqnarray*}
		From this, $p\in \partial^+ (f*g)(x)$ and \eqref{eq:1} it follows that
		\begin{eqnarray*}
      p(h)  &\ge& f*g(x+h) - f*g(x)\\
                &\ge& (f(x-y) + g(y+h)) - (f(x-y)+g(y) + \e)\\
                &\ge& g(y+h) - g(y) - \e.
    \end{eqnarray*}
		That is, $p\in \partial^+_{\e}g(y)$.
	\end{proof}
	\begin{prop}\label{prop:psi}
	        Let $A$ and $B$ be convex subsets of the Banach space $X$. Let $r,s\in\mathbb{R}$ be such that $r\neq s$.
    We define the concave function $\psi$ by
		\begin{equation}\label{def:psi}
      \hyp \psi := \co\{ A\times [-\infty;r],\ B\times [-\infty;s]\}.
    \end{equation}
		Let $x_0\in\dom\psi = [A,B]$ be such that $\psi(x_0)\neq s$.
    Let $p\in\epar^+ \psi (x_0)$. Then
		\begin{equation}\label{eq:prop:psi}
      \inf _A p - \inf_B p \le r - s + \frac{r-s}{\psi(x_0)-s}\e.
    \end{equation}
	\end{prop}
	\begin{proof}
		Let $p\in\epar^+\psi(x_0)$. Set $l(x):=\psi(x_0)+p(x-x_0)$. Note that by definition we have $l\ge\psi-\e$. First we note that as $(x_0,\psi(x_0))\in\hyp\psi$ (the reader is advised to draw a simple picture of $\psi$), we can find points $(u,\bar{r})\in A\times[-\infty,r]$, $(v,\bar{s})\in B\times[-\infty,s]$ such that:
		$$
		  (x_0,\psi(x_0)) = \lambda(u,\bar{r}) + (1-\lambda)(v,\bar{s}),
		$$
		for some $\lambda\in[0;1]$. Suppose that $\lambda=0$. Then $(x_0,\psi(x_0))=(v,\bar{s})$, or $\psi(x_0)=\psi(v)=\bar{s}\le s$. On the other hand, $\psi(v)\ge s$. So, $\psi(x_0)=s$, a contradiction to $\psi(x_0)\neq s$. Thus, $\lambda\in(0;1]$. Furthermore, as
		$$
			x_0=\lambda u+(1-\lambda)v; \ \psi(x_0)=\lambda\bar{r}+(1-\lambda)\bar{s}
		$$
		and $\bar{r}\le r$ and $\bar{s}\le s$, we have that
		$$
			\psi(x_0)\le\lambda r+(1-\lambda)s
		$$
		and at the same time $(x_0,\lambda r+(1-\lambda)s)=\lambda(u,r)+(1-\lambda)(v,s)\in\hyp\psi$. It follows that
		$$
			\lambda\bar{r}+(1-\lambda)\bar{s}=\lambda r+(1-\lambda)s, \ \text{or equivalently}
		$$
		$$
			\lambda(r-\bar{r})+(1-\lambda)(s-\bar{s})=0.
		$$
		There are two cases:\\
	Case 1: $\lambda\in(0;1)$. In this case we have that $r=\bar{r}$ and $s=\bar{s}$, or
	  \begin{equation}\label{eq:psi_representation}
			(x_0,\psi(x_0))=(\lambda u+(1-\lambda)v,\lambda r+(1-\lambda)s)=\lambda (u, r) + (1-\lambda) (v, s)
		\end{equation}
	It follows that
		$$
			p(v-x_0)\ge\psi(v)-\psi(x_0)-\e\ge s-\psi(x_0)-\e=\lambda(s-r)-\e.
		$$
	As we have
		$$
			u-x_0=-\frac{1-\lambda}{\lambda}(v-x_0),
		$$
	we get
		$$
			p(u-x_0)\le\frac{1-\lambda}{\lambda}[-\lambda(s-r)+\e]=(\lambda-1)(s-r)+\frac{1-\lambda}{\lambda}\e.
		$$
	It follows that
		\begin{eqnarray*}
		  \inf_Al&\le& l(u)=\psi(x_0)+p(u-x_0)\\
		        &\le&\psi(x_0)+(\lambda-1)(s-r)+\frac{1-\lambda}{\lambda}\e=r+\frac{1-\lambda}{\lambda}\e  .
		\end{eqnarray*}
	On the other hand, as $l\ge\psi-\e$, we have
		$$
			\inf_Bl\ge\inf_B(\psi-\e)=\inf_B\psi-\e \ge s-\e
		$$
	and after combining the last two inequalities, we obtain
		$$
			\inf_Bp-\inf_Ap=\inf_Bl-\inf_Al\ge s-\e-r-\frac{1-\lambda}{\lambda}\e=s-r-\frac{\e}{\lambda}.
		$$
	Finally, $s-\psi(x_0)=\lambda(s-r)$, so $\lambda^{-1}=(s-r)/(s-\psi(x_0))$
    and we get \eqref{eq:prop:psi}.\\
	Case 2: $\lambda=1$. Here $r=\bar{r}$. Then $(x_0,\psi(x_0))=(u,r)$. We have
		$$
			\inf_Al\le l(u)=\psi(x_0)+p(u-x_0)=r+p(0)=r,
		$$
		$$
			\inf_Bl\ge\inf_B\psi-\e \ge s-\e
		$$
	Here we get
		$$
			\inf_Bp-\inf_Ap=\inf_Bl-\inf_Al\ge s-r-\e.
		$$
	Since $s-\psi(x_0)=s-\psi(u)=s-r$, this is equivalent to \eqref{eq:prop:psi}.
	\end{proof}
	The following function plays in our proof the role of the linear function in the standard proof of Lagrange Mean Value Theorem.
	\begin{prop}\label{prop:fi}
	Let $A$ and $B$ be convex subsets of the Banach space $X$ and let $r,s\in\mathbb{R}$ be such that $r\neq s$. We consider the function $\psi$ as defined in Proposition~\ref{prop:psi}. Let $K>0$ and
		$$
      \ffi_K(x) :=(-K\|\cdot\|)*\psi(x)= \sup\{\psi(y)-K\|x-y\|:\ y\in X\}.
    $$
  Then $\ffi_K$ is $K$-Lipschitz and concave.

	Let $\bar x$ be such that there exists $c > 0$ for which the sets
  $$
    U:=\{z\in [A,B]:\ \psi(z)-K\|z-\bar x\| > \ffi_K (\bar x) - c\}
  $$
  and
  $$
    V:=\{z\in [A,B]:\ |s-\psi(z)|\} < c
  $$
  do not intersect: $U\cap V=\emptyset$.

  If $p\in\partial^+ \ffi_K(\bar x)$ then
  \begin{equation}\label{prop:fi:conclusion}
    \inf_B p - \inf_A p \ge s-r.
  \end{equation}
	\end{prop}
	\begin{proof}
    Since $\psi$ is bounded from above, $\ffi_K$ is well defined.

    As a sup-convolution of two concave functions $\ffi_K$ is itself concave, see for example \cite[p. 41]{penot}.

    As $-K\|\cdot\|$ is a $K$-Lipschitz function, it easily follows that $\ffi_K$ is also $K$-Lipschitz.

    For the main part note that from $U\cap V=\emptyset$ it easily follows that for any sequence $(x_n)_1^\infty\subseteq [A,B]$ such that
  		\begin{equation}\label{eq:max_sec_1}
        \ffi_K(\bar x) = \lim_{n\to\infty} (\psi(x_n)-K\|\bar{x}-x_n\|)
      \end{equation}
  	it holds
  		\begin{equation}\label{eq:max_sec_2}
  			|s-\psi(x_n)|\ge c, \quad \forall n\mbox{ large enough}.
  		\end{equation}
      We can assume that the latter is fulfilled for all $n$.

		The definition of $\ffi_K$ is equivalent to:
			$$
				\ffi_K(x)=\sup\{\psi(y)-K\|x-y\|:\ y\in[A,B]\},
			$$
			since $\psi=-\infty$ outside $[A,B]$. From \eqref{eq:max_sec_1} it follows that
			we can find $\e_n\downarrow0$ with
			$$
        \psi(x_n)-K\|\bar{x}-x_n\|\le\ffi_K(x_n) < \psi(x_n)-K\|\bar{x}-x_n\| + \e_n.
      $$
			In particular, we have:
			$$
				(-K\|\cdot\|)*\psi(\bar{x})-\e_n<-K\|\bar{x}-x_n\|+\psi(x_n).
			$$
			For $p\in\partial^+\ffi_K(\bar{x})$ we apply Lemma~\ref{sup_conv}. It follows
      that $p\in \partial^+_{\e_n}\psi(x_n)$. So, from Proposition~\ref{prop:psi} and \eqref{eq:max_sec_2} we get
			$$
        \inf_A p - \inf _B p \le r-s + \frac{r-s}{\psi(x_n)-s}\e_n \le
        r-s + \frac{|s-r|}{c}\e_n.
      $$
			Since $\e_n\downarrow 0$, we are done.
  \end{proof}

  We will also need few more preparatory claims.
	\begin{lemm}\label{lem:s}
   Let $A$, $B$ be convex and bounded subsets of the Banach space $X$. Let $r\neq s\in\mathbb{R}$ and let $\psi$ be constructed as in Proposition~\ref{prop:psi}. If the sequence $(x_n)_{n=1}^\infty$ is such that $\psi(x_n)\to s$ as $n\to\infty$, then $d(x_n,B)\to 0$.
  \end{lemm}
  \begin{proof}
    Let $(x_n,\psi(x_n))=\lambda_n(u_n,r)+(1-\lambda_n)(v_n,s)$, for some $u_n\in A$, $v_n\in B$ and $\lambda_n\in[0,1]$.

    From $r\neq s$ and $\lambda_nr+(1-\lambda_n)s\to s$ it immediately follows that $\lambda_n\to 0$.

    But $d(x_n,B)\le \|x_n-v_n\|=\lambda_n\|u_n-v_n\|\to 0$, since $A$ and $B$ are bounded.
  \end{proof}
  \begin{lemm}\label{lem:btwsr}
    Let $A$, $B$ be convex and bounded subsets of the Banach space $X$. Let $r\neq s\in\mathbb{R}$ and let $\ffi_K$ be constructed as in Proposition~\ref{prop:fi} for some $K>0$. Then
    \begin{equation}\label{eq:btwsr}
      \min\{r,s\} \le \ffi_K(x) \le \max\{r,s\},\quad\forall x\in [A,B].
    \end{equation}
  \end{lemm}
  \begin{proof}
    Since
    \begin{eqnarray*}
      (A\cup B)\times (-\infty,\min\{r,s\}]&\subseteq& A\times [-\infty;r]\cup B\times [-\infty;s]\\
      &\subseteq& (A\cup B)\times (-\infty,\max\{r,s\}],
    \end{eqnarray*}
    taking convex envelopes we get
    $$
      [A,B]\times (-\infty,\min\{r,s\}] \subseteq \hyp \psi \subseteq [A,B]\times (-\infty,\max\{r,s\}],
    $$
    or, in other words, $\min\{r,s\}\le \psi \le \max\{r,s\}$ on $[A,B]$.

    Since $\psi\le\max\{r,s\}$, from the definition of sup-convolution it readily follows that $\ffi_K\le\max\{r,s\}$.

    On the other hand, if $x\in [A,B]$ then $\ffi_K(x) \ge \psi(x)\ge\min\{r,s\}$.
  \end{proof}
	\begin{lemm}\label{lem:d}
	 Let $A$, $B$ be convex subsets of the Banach space $X$ and $\delta>0$. For the set $C=\overline{[A,B]_\delta}$ we have that its topological boundary $\partial C$ satisfies
	$$\partial C\subseteq\{x\in C:d(x,[A,B])=\delta\}.$$
	\end{lemm}
	\begin{proof}
	 First note that since $C$ is closed, we have $\partial C=C\backslash\mathrm{int}C$, where the latter denotes the topological interior of $C$. Observe that
	$$
	   \forall x: \ d(x,[A,B])>\delta \implies x\notin C.
	$$
	Indeed, if $d(x,[A,B])>\delta$ for some $x\in X$ then $\exists\e>0$ with $d(x,[A,B])>\delta+\e$. Then for each $y\in B(x,\e):=\{y\in X:\|y-x\|<\e\}$
	$$\|y-z\|\ge\|z-x\|-\|y-x\| > \|z-x\| - \e, \quad \forall z\in X.$$
	This implies
	$$\inf_{[A,B]}\|y-\cdot\|\ge\inf_{[A,B]}\|x-\cdot\|-\e>\delta+\e-\e=\delta,$$
	which means that $[A,B]_{\delta}\cap B(x,\e)=\emptyset$. Thus $x\notin C$.
	Next, it is obvious that
	$$
	  \forall x: \ d(x,[A,B])<\delta \implies x\in\mathrm{int}C.
	$$
	\end{proof}
  \section{Proof of the main result}\label{sec:proof}
		In this section we prove Theorem~\ref{thm:cl}.

    Fix $\e>0$.

    Fix $s_1$ such that $s<s_1<s+\min\{\e,\e\delta\}$, $s_1<\inf_{B_\delta} f$ and $s_1\neq r$. Note that
    \begin{equation}\label{eq:modsr}
      |r-s_1| < |r-s| + \e.
    \end{equation}
    Also,
		$$
      \frac{\max\{r,s_1\}-\mu}{\delta} < \frac{\max\{r,s\}-\mu}{\delta}+\e.
    $$
    Take $\delta_1\in(0,\delta)$ such that
    \begin{equation}\label{eq:def_K}
      K:= \frac{\max\{r,s_1\}-\mu}{\delta_1} < \frac{\max\{r,s\}-\mu}{\delta}+\e.
    \end{equation}
    Let $\ffi_K$ be the function constructed in Proposition~\ref{prop:fi} with these $r$, $s_1$ and $K$.

    By Lemma~\ref{lem:d} we have $\pa C\subseteq\{x\in C:\  d(x,[A,B])=\delta\}$. If $x\in \pa C$ then $\ffi_K(x)\le \sup\{\psi(y):\ y\in [A,B]\} + \sup \{-K\|y-x\|:\ y\in [A,B]\} \le \max\{r,s_1\} -K\inf \{\|y-x\|:\ y\in [A,B]\}=\max\{r,s_1\} -K\delta < \mu$. That is,
		\begin{equation}\label{eq:boundC}
      \forall x\in \pa C\Rightarrow \ffi_K(x)< \mu.
    \end{equation}
		Set
    $$
      f_1(x):= \begin{cases}
        f(x),& x\in C,\\
        \infty,& x\not\in C.
      \end{cases}
    $$
    Since $C$ is closed, $f_1$ is lower semicontinuous. Also, $\inf f_1  > \mu$ from \eqref{eq:D}.

    From (P2) we have that $\pa f_1(x) = \pa f(x)$ for $x\in C\setminus \pa C$.

    Consider
    $$
      g(x):= f_1(x) - \ffi_K(x).
    $$
     and note that $\dom f_1=\dom g\subseteq C$. From the above and \eqref{eq:boundC} we have that the lower semicontinuous function $g$ ($\ffi_K$ is $K$-Lipschitz, Proposition~\ref{prop:fi}) is bounded below and, moreover,
		\begin{equation}\label{eq:bound_g}
      \inf\{g(x):\ x\in\pa C\} > 0.
    \end{equation}
		We claim that
		\begin{equation}\label{eq:inf_g}
      \inf g \le 0.
    \end{equation}

		Indeed, from \eqref{eq:rs} for any $t > 0$ there is $x\in A$ such that $f_1(x)=f(x)<r+t$. On the other hand, $\ffi_K(x)\ge \psi (x)\ge r$ by the very construction of $\psi$, see \eqref{def:psi}. Therefore, $g(x) < t$.

    Since $-\ffi_K$ is convex and continuous, we can apply $(P4')$ from Proposition~\ref{p4p} to $f_1$ and$-\ffi_K$ to get
    $$
      p_n\in\pa f_1(x_n)\mbox{ and }q_n\in\pa (-\ffi_K)(y_n)=-\pa^+\ffi_K(y_n),
    $$
    such that
    \begin{equation}\label{eq:3cond}
      \|x_n-y_n\|\to 0,\ (f_1(x_n){-}\ffi_K(y_n))\to\inf(f_1{-}\ffi_K),\ \|p_n+q_n\|\to 0.
    \end{equation}
    We will next show that for all $n\in\mathbb{N}$ large enough $(\xi,p)=(x_n,p_n)$ satisfies the conclusions of Theorem~\ref{thm:cl}.
    \begin{lemm}\label{lem:step1}
      $x_n\in \mathrm{int} C$ for all $n\in\mathbb{N}$ large enough and, therefore, $p_n\in\pa f(x_n)$.
    \end{lemm}
    \begin{proof}
      Note that $\inf(f_1{-}\ffi_K) = \inf g \le 0$ by \eqref{eq:inf_g}, so $\lim (f_1(x_n){-}\ffi_K(y_n)) \le 0$ by \eqref{eq:3cond}.

      Assume that there exists subsequense $(x_{n_i})_{i=1}^\infty\subseteq\pa C$. Then \\$(f_1(x_{n_i}){-}\ffi_K(y_{n_i}))
      = g(x_{n_i})- (\ffi_K(x_{n_i})-\ffi_K(y_{n_i}))\ge \inf _{\pa C} g - K\|x_{n_i}-y_{n_i}\|$, since $\ffi_K$ is $K$-Lipschitz (see Proposition~\ref{prop:fi}). From \eqref{eq:bound_g} and \eqref{eq:3cond} it follows that the latter tends to strictly positive limit. Contradiction.
    \end{proof}

    The estimate \eqref{eq:norm_p} is easy to check: from \eqref{eq:3cond} and the $K$-Lipschitz continuity of $\ffi_K$,
    which implies $\|q_n\|\le K$, it follows that $\limsup\|p_n\|\le K$ and we need only recall \eqref{eq:def_K}.
    \begin{lemm}
      For all $n\in\mathbb{N}$ large enough
      $$
        f(x_n) < \inf_{[A,B]}f+|r-s|+\e.
      $$
    \end{lemm}
    \begin{proof}
      Let $\nu:= |r-s|+\e-|r-s_1|$. From \eqref{eq:modsr} we have $\nu > 0$.

      As in the proof of Lemma~\ref{lem:step1} we use the Lipschitz continuity of $\ffi_K$ to see that for all $n$ large enough
      $$
        f(x_n) -\ffi_K(y_n) < \inf\{ f(x)-\ffi_K(x):\ x\in [A,B]\} + \nu.
      $$
      But from \eqref{eq:btwsr} we have $\min\{r,s_1\}\le \ffi _K \le \max\{r,s_1\}$ on $[A,B]$. Therefore,
      $$
        f(x_n) - \max\{r,s_1\} < \inf _{[A,B]} f - \min\{r,s_1\} + \nu.
      $$
      Obviously, $\max\{r,s_1\} - \min\{r,s_1\} = |r-s_1|$ and, therefore,
      $f(x_n) < \inf _{[A,B]} f + |r-s_1| + \nu = \inf _{[A,B]} f + |r-s| + \e$.
    \end{proof}
    \begin{lemm}\label{lem:seq}
      For all $n\in\mathbb{N}$ large enough there exist $c_n>0$ such that the sets
	$$
    	  U_{n,c_n}:=\{z\in [A,B]:\ \psi(z)-K\|z-y_n\| > \ffi_K (y_n) - c_n\}
  	$$
  	and
  	$$
    	  V_{c_n}:=\{z\in [A,B]:\ |s_1-\psi(z)| < c_n\}
  	$$
  do not intersect, that is $U_{n,c_n}\cap V_{c_n}=\emptyset$.
    \end{lemm}
    \begin{proof}
      Fix $\bar \e > 0$ such that
      \begin{equation}\label{eq:bareps}
        \bar\e < \inf_Cf-\mu,\quad \bar\e < \inf_{B_\delta}f-s_1,\quad \bar \e < \frac{\delta-\delta_1}{1+K^{-1}}.
      \end{equation}
      Let $n$ be so large that for $\bar{x}=x_n$ and $\bar{y}=y_n$ it is fulfilled
      \begin{equation}\label{eq:barxy}
  			\|\bar x - \bar y\| <   \bar \e,\ f_1(\bar{x})-\ffi_K(\bar{y})<\bar{\e},
  		\end{equation}
      see \eqref{eq:3cond}.

	For this fixed $n$ assume the contrary, that is, for any positive $c_n >0$ the sets $U_{n,c_n}$ and $V_{c_n}$ defined with this $c_n$, intersect.

  For any $m\in\mathbb{N}$ chose $z_m\in U_{n,m^{-1}}\cap V_{m^{-1}}$.

	Then the sequence $\{z_m\}\subseteq [A,B]$ satisfies
      \begin{equation}\label{eq:min_seq}
        \ffi_K(\bar{y}) = \lim_{m\to\infty}(\psi(z_m)-K\|z_m-\bar{y}\|),
      \        \lim_{m\to\infty} \psi(z_m) = s_1.
      \end{equation}
      It follows that $\ffi_K(\bar{y}) = s_1 -K\lim_{m\to\infty}\|z_m-\bar{y}\|$.
      But by \eqref{eq:barxy} and \eqref{eq:D} we have $\ffi_K(\bar{y}) > f_1(\bar{x}) - \bar \e > \mu - \bar \e$.
      So,  $K\lim_{m\to\infty}\|z_m-\bar{y}\| < s_1 -\mu + \bar \e \le K\delta_1 +\bar \e$ from \eqref{eq:def_K}. But
       Lemma~\ref{lem:s} implies $d(z_m,B)\to 0$, thus $d(\bar y, B) \le \delta_1 + \bar \e/K < \delta - \bar \e$ from \eqref{eq:bareps}.

       From this, \eqref{eq:barxy} and triangle inequality it follows that $d(\bar x, B)<\delta$. Therefore,
       $\bar x\in B_\delta$ and $f_1(\bar x) \ge \inf_{B_\delta} f$. But from \eqref{eq:min_seq} it is obvious that
       $\ffi_K(\bar y) \le s_1$ and, therefore, $f_1(\bar x) - \ffi_K(\bar y) \ge \inf_{B_\delta} f - s_1 >\bar \e$
       from \eqref{eq:bareps}. This, however, contradicts \eqref{eq:barxy}.
    \end{proof}
    From Proposition~\ref{prop:fi} and Lemma~\ref{lem:seq} it follows that
    $$
      \inf_B(-q_n)-\inf_A(-q_n)\ge s_1-r
    $$
    for all $n$ large enough. Since $A$ and $B$ are bounded, from \eqref{eq:3cond} we get
    $$
      |\inf_Ap_n-\inf_A(-q_n)| \le \|p_n+q_n\|\sup_{X\in A}\|x\|\to 0,
    $$
    $$
      |\inf_Bp_n-\inf_B(-q_n)| \le \|p_n+q_n\|\sup_{x\in B}\|x\|\to 0.
    $$
    From the three above and $s_1>s$ it follows that
    $$
      \inf_B p_n-\inf_A p_n > s-r
    $$
    for all $n$ large enough.

    \textit{The proof of Theorem~\ref{thm:cl} is thus completed.}

  \section*{Acknowledgment.} The second named author would like to thank Prof. R. Deville for poining out the potential for exploration of then new Clarke-Ledyaev inequality. Thanks are also due to Prof. N. Zlateva for valuable suggestions.

 \end{document}